\documentclass[10pt]{amsart}

\usepackage{amssymb,amsfonts,amsthm,amscd,stmaryrd}

\usepackage{color}
\usepackage{ae}
\usepackage[utf8]{inputenc}
\numberwithin{figure}{section}

\usepackage{mathrsfs,a4wide}
\usepackage{esint}

\usepackage{import}

\usepackage[final,allcolors=blue,colorlinks=true]{hyperref}
\usepackage[leqno]{amsmath}

\numberwithin{figure}{section}

\newtheorem{theorem}{Theorem}[section]
\newtheorem{lemma}[theorem]{Lemma}

\theoremstyle{definition}

\numberwithin{equation}{section}
\newcommand{\R}{\mathbb{R}}

\newcommand{\eps}{\varepsilon}

\newcommand{\pot}{\gamma}

\newcommand{\Ha}{{\mathcal{H}}}

\begin{document}

\title[Isoperimetric problem with a Coulombic repulsive term]{Isoperimetric problem with a Coulombic repulsive term}

\author{Vesa Julin}

\address{Department of Mathematics and Statistics, University of Jyv\"askyl\"a, Finland}

\thanks{}

\subjclass[2000]{}
%\date{\today}

\begin{abstract} We consider a non-local isoperimetric problem with a repulsive Coulombic term. In dimension three this corresponds to the Gamow's famous  liquid drop model. We show that whenever the mass is small the ball is the unique minimizer of the problem. The proof is based on a strong version of a quantitative isoperimetric inequality introduced in \cite{FJ}.
\end{abstract}
\maketitle

\section{Introduction}

We study a ground state problem with a volume constraint
\begin{equation}
\label{energy}
\min \left( P(F) + \int_{\R^n}\int_{\R^n} \frac{\chi_F(x)\chi_F(y)}{|x-y|^{n-2}} \, dy\,  dx   \!: \,\, |F|=m \right),
\end{equation}
where $F \subset \R^n$ is a set of finite perimeter and $\chi_F$  denotes its characteristic function. Here $P(F)$ is the perimeter of $F$ and the  term with the Newton potential is called non-local part of the energy. It acts as a repulsive force and prefers  sets to be disconnected. 

The most important case is $n=3$, which corresponds to  Gamow's famous liquid drop model \cite{Ga}. It was introduced to model the stability of the atomic nucleus and the atomic fission. The nature of the problem \eqref{energy} is very fundamental and it appears  in many physical phenomena. For more about the physical background and further references see  \cite{Mu}, where the  general dimensional case is considered.

In the absence of the repulsive non-local part of the energy, the above problem is the isoperimetric problem. It is well known that then the unique minimizer, up to a translation, is the ball, and it is therefore plausible that the ball remains the minimizer of \eqref{energy} when the repulsive part of the energy is small. However, due to the different behaviour of the two  competing energy terms, the problem is very delicate and even the existence of a minimizer is highly non-trivial.   

The goal is to prove the following stability result for small masses.
\begin{theorem}
\label{mainthm}
Suppose that $n\geq 3$. There exists $m_n>0$, depending only on the dimension $n$, such that for every $m < m_n$ the ball $B_{r_m}$ is, up to a translation, the unique minimizer of \eqref{energy}.
\end{theorem}

The above result was proven in a very recent paper by Kn\"upfer and Muratov \cite{KnMu2} in dimensions $n \leq 7$. There the non-local part of the energy is allowed to have a general Riesz potential structure $\int_{\R^n}\int_{\R^n} \frac{\chi_F(x)\chi_F(y)}{|x-y|^{\alpha}} \, dy\,  dx$  with $\alpha< n -1 $, which covers the case of the Newton potential. They also considered  the problem of unstability for the large masses. A somewhat similar approach to a slightly different problem can be found in \cite{CiSp}. 

We stress that this paper was done independently from  \cite{KnMu2}. Our approach is quite different and it is based on  a new quantitative isoperimetric inequality introduced in \cite{FJ}. In fact, the most important case $n=3$ is nothing but a direct corollary of Theorem \ref{SQII} stated below. We  need some extra work in the higher dimensional cases but the arguments remain standard. We choose to deal only with the  Newton potential, to avoid the simplicity of our idea getting buried under technicalities.  Moreover, Theorem \ref{mainthm} is, to  best of  the author's knowledge, the only result in higher dimensions.

Finally we note that the planar case was studied extensively in \cite{KnMu1} with a general Riesz potential in the non-local part of the  energy.  We remark that  Theorem \ref{mainthm} is not true in the plane, i.e., when the non-local part of the energy scales as the Newton potential. This is due to the fact that in the plane  the Newton potential behaves as $\log (1/r)$. Since $\log (1/r) \to - \infty$ as $r \to \infty$ one may consider a  family of sets consisting on two disks which are further and further from each other. It is then obvious, that  the value of the problem corresponding to \eqref{energy}, would always be $-\infty$.

\section{Proofs}

By scaling we may write the problem  as
\begin{equation}
\label{scale}
\min \left( P(E) + \lambda_m \int_{\R^n}\int_{\R^n} \frac{\chi_E(x)\chi_E(y)}{|x-y|^{n-2}} \, dy\,  dx \!:\,\,  |E|= |B_1| \right),
\end{equation}
where $\lambda_m =\left( \frac{m}{|B_1|} \right)^{\frac{3}{n}}$ and $m$ is the mass in the original problem. In future we prefer the sets to  have the measure of the unit ball and therefore we consider the problem \eqref{energy}  in  scaled form \eqref{scale}.

As mentioned in the introduction, the proof is based on a new type of quantitative isoperimetric inequality. Suppose that we have a set of finite perimeter $E$ such that $|E|= |B_1|$. We measure the oscillation of the boundary of $E$ by the assymmetry
\[
\beta(E) := \min_{y \in\R^n}\biggl\{  \biggr( \frac{1}{2}\int_{\partial^*E}|\nu_E(x)-\nu_{B_1(y)}(\pi_{y}(x))|^2\,d\Ha^{n-1}(x)\biggr)^{1/2} \biggr\}\,,
\] 
where $\nu_{B_1(y)}(\pi_{y}(x)) = \frac{x-y}{|x-y|}$ and $\partial^*E$ is the reduced boundary of $E$. For precise definitions and properties of the sets of finite perimeter we refer to \cite{AFP}. The quantitative isoperimetric inequality in strong form reads as follows.
\begin{theorem}
\label{SQII}
Suppose $n \geq 2$. There is a dimensional constant $c_n$ such that for every  set of finite perimeter $E \subset \R^n$ with $|E|= |B_1|$ it holds
\[
 P(E) - P(B_1)\, \geq c_n \, \beta(E)^2.
\]
\end{theorem}

Notice that by the divergence theorem we may write 
\[
\begin{split}
\frac{1}{2} \int_{\partial^*E}&|\nu_E(x)-\nu_{B_1(y)}(\pi_{y}(x))|^2\, d\Ha^{n-1}(x) = \frac{1}{2} \int_{\partial^*E}\big|\nu_E(x)-\frac{x-y }{|x-y |}\big|^2\,d\Ha^{n-1}(x) \\
&=   \int_{\partial^*E}\left( 1- \nu_E(x) \cdot \frac{x-y }{|x-y |}\right) \,d\Ha^{n-1}(x) \\
&=  P(E)   - \int_{E}\frac{n-1 }{|x-y |} \,d\Ha^{n-1}(x)  \\
&=  P(E)   -P(B_1)  +  \int_{B_1(y)}\frac{n-1 }{|x-y |} \,d\Ha^{n-1}(x) - \int_{E}  \frac{n-1 }{|x-y |} \,d\Ha^{n-1}(x) \, .
\end{split}
\]
Therefore by defining an  assymmetry  
\begin{equation}
\label{gamma}
 \gamma(E) :=\min_{y\in\R^n} \left( \int_{B_1(y)} \frac{n-1}{|x-y|} dx - \int_{E} \frac{n-1}{|x-y|} dx  \right) ,
\end{equation}
Theorem \ref{SQII} can be written as
\begin{equation}
\label{pot_form}
 P(E) - P(B_1)\, \geq C_n \, \gamma(E),
\end{equation}
for every set of finite perimeter $E \subset \R^n$ with $|E| = |B_1|$, where $C_n = \frac{c_n}{1-c_n}$.

We are ready to prove Theorem \ref{mainthm}. The proof in dimension three turns out to be the easiest. This is simply due to the fact that the asymmetry \eqref{gamma} scales exactly as the Newton potential. The proof also gives a nice bound for the critical mass $m_3 \geq \frac{1}{C_3}$ where $C_3$ is the constant from \eqref{pot_form}. Unfortunately the proof of \eqref{pot_form} in \cite{FJ} doesn't give any explicit bound for the constant $C_3$. The higher dimensional cases are proven simlarly by using Lemma \ref{induction} iteratively. 

\begin{proof}[\textbf{Proof of the Theorem \ref{mainthm} in the case $n=3$.}]
We begin the  proof in a general dimension $n\geq 3$ and specify to the case $n=3$ only at the end. 

Consider the problem in the scaled form \eqref{scale}. Suppose $E \subset \R^n$ is such that $|E| = |B_1|$. Since the problem is translation invariant we may assume  $E$ is centered at the origin, i.e., 
\[
\gamma(E) = \int_{B_1} \frac{n-1}{|x|} dx - \int_{E} \frac{n-1}{|x|} dx.
\] 
Write $\text{NL}(E) := \int_{\R^n}\int_{\R^n} \frac{\chi_E(x)\chi_E(y)}{|x-y|^{n-2}} \, dy\,  dx$. We need to show that
\[
 \text{NL}(B_1) -\text{NL}(E) \leq C\, (P(E)- P(B_1)) 
\]  
for some dimensional constant $C>0$. We organize the terms 
\begin{equation}
\label{compare}
\begin{split}
 \text{NL}(B_1) &-\text{NL}(E) \\
&=  \int_{\R^n}\int_{\R^n}\frac{\chi_{B_1}(x)\chi_{B_1}(y)}{|x-y|^{n-2}} \, dy\,  dx- \int_{\R^n}\int_{\R^n}\frac{\chi_E(x)\chi_E(y)}{|x-y|^{n-2}}\, dy\,  dx  \\
&=2  \int_{\R^n}\int_{\R^n} \frac{\chi_{B_1}(y) ( \chi_{B_1}- \chi_{E})(x) }{|x-y|^{n-2}} \, dx\,  dy \\
&\qquad - \int_{\R^n}\int_{\R^n} \frac{ (\chi_{B_1} - \chi_{E})(y) (\chi_{B_1} - \chi_{E})(x)}{|x-y|^{n-2}}  \, dx\,  dy .
\end{split}
\end{equation}

Define 
\[
v(y) = \int_{\R^n} \frac{ (\chi_{B_1} - \chi_{E})(x)}{|x-y|^{n-2}}  \, dx,
\]
which solves 
\[
- \Delta v =  \tilde{c}_n \, (\chi_{B_1} - \chi_{E})
\]
for some dimensional constant $ \tilde{c}_n >0$. Therefore 
\[
\begin{split}
 &\int_{\R^n}\int_{\R^n} \frac{ (\chi_{B_1} - \chi_{E})(y) (\chi_{B_1} - \chi_{E})(x)}{|x-y|^{n-2}}  \, dx\,  dy \\
&= - \frac{1}{ \tilde{c}_n } \int_{\R^3} v(y)  \, \Delta v(y)\,  dy \\
&= \frac{1}{ \tilde{c}_n}   \int_{\R^3} | \nabla v|^2\,  dy \geq 0.
\end{split}
\]
This and \eqref{compare}  yield
\begin{equation}
\label{bound1}
 \text{NL}(B_1) - \text{NL}(E) \leq 2  \int_{\R^n}\int_{\R^n} \frac{\chi_{B_1}(y) ( \chi_{B_1}- \chi_{E})(x) }{|x-y|^{n-2}} \, dx\,  dy = 2 \int_{B_1} v(y)\, dy.
\end{equation}

Assume now that $n=3$. Then $v(y)=  \int_{\R^n} \frac{(\chi_{B_1} - \chi_{E})(x)}{|x-y|} \, dx$. Since $v$ solves $- \Delta v =  \tilde{c}_3\, (\chi_{B_1} - \chi_{E}),$ it is superharmonic in $B_1$ and therefore, by the mean value property, we conclude that 
\[
2 \fint_{B_1} v(y)\, dy \leq 2 v(0)= 2 \int_{\R^n} \frac{(\chi_{B_1} - \chi_{E})(x)}{|x|} \, dx = \gamma(E).
\]
Hence, by \eqref{bound1} and  \eqref{pot_form}, we have
\[
\text{NL}(B_1) -\text{NL}(E)  = |B_1|  \, \gamma(E) \leq  \frac{|B_1| }{C_3} \, (P(E) - P(B_1))
\]
which concludes the proof in the case $n=3$.

\end{proof}

The above proof contains all the relevant arguments for the higher dimensional case. Only  the last step, where we used the mean value property, doesn't  generalize since the inequality of the type
\[
v(0) =  \int_{\R^n} \frac{(\chi_{B_1} - \chi_{E})(x)}{|x|^{n-2}} \lesssim  P(E) - P(B_1)
\]
is not true when $n \geq 4$.  Here the notation ''$\lesssim$'' means that the above inequality is true up to a multiplication with a positive dimensional constant. This can be easily  seen by considering a family of annuli $E_{\eps} = B_{R_{\eps}} \setminus B_{\eps}$,  where $R_{\eps}= \sqrt[n]{1 + \eps^n}$. We need to execute slightly more careful analysis. To that end we recall some well known definitions and results.

Suppose that $f : \R^n \to \R$ is a locally integrable function. The Riesz potential $I_{\alpha}f$ of $f$, for $\alpha \in (0, n)$, is defined by
\begin{equation}
\label{Riesz}
I_{\alpha}f(x):=\sigma_{n , \alpha} \int_{\R^n} \frac{f(\xi)}{|x- \xi|^{n-\alpha}} \, d\xi,
\end{equation}
where the normalization constant is given by
\[
\frac{1}{\sigma_{n , \alpha}}  = \pi^{n/2} 2^{\alpha} \frac{\Gamma \left( \alpha/2\right)}{\Gamma \left( (n -\alpha)/2\right)}.
\]
For the Riesz potential we have the semigroup property
\begin{equation}
\label{semigroup}
I_{\alpha_1}(I_{\alpha_2} f) =  I_{\alpha_1 + \alpha_2} f  \qquad \text{if }\, \alpha_1, \alpha_2 \in (0, n)\, ,\text{s.t.} \quad \alpha_1 + \alpha_2 < n 
\end{equation}
and
\begin{equation}
\label{Rieszlaplace}
- \Delta (I_{\alpha +2}f )= I_{\alpha}f \qquad \text{if }\, \alpha \in (0, n-2).
\end{equation}

We need the following simple  lemma.

\begin{lemma}
\label{induction}
Suppose that $n\geq 5$, $\alpha \in [2, n-3]$ and $E \subset \R^n$ is a set of finite perimeter with $|E|= |B_1|$. Denote  $f_E = \chi_{B_1} - \chi_{E}$. If the function $\phi_{\alpha} : (0, \infty) \to \R$
\[
\phi_{\alpha}(r) := \fint_{\partial B_r} (I_{\alpha}f_E)(x) \, d \Ha^{n-1}(x)
\]
is decreasing, then it holds that
\begin{enumerate}
\item[(i)]  $\phi_{\alpha+2} : (0, \infty) \to \R$
\[
\phi_{\alpha+2} (r) := \fint_{\partial B_r} (I_{\alpha+2}f_E)(x) \, d \Ha^{n-1}(x)
\]
is decreasing,
\item[(ii)] 
\[
\fint_{B_1} (I_{\alpha}f_E)(x) \, dx \leq 2 n(n+2)\fint_{B_1} (I_{\alpha+2}f_E)(x) \, dx,
\]
\item[(iii)] and
\[
\fint_{B_1} (I_{\alpha}f_E)(x) \, dy \leq 2n \,  (I_{\alpha +2}f_E)(0) .
\]
\end{enumerate}
\end{lemma}

\begin{proof}
The proof is standard. Define $\Phi_{\alpha}: [0, \infty) \to \R$ as 
\[
\Phi_{\alpha}(r) :=  \fint_{B_r}  (I_{\alpha}f_E)(x) \, dx .
\]
Then it holds
\[
\Phi_{\alpha}(r) = \int_0^r \int_{\partial B_s}    (I_{\alpha}f_E)(x) \, d \Ha^{n-1}(x) ds = n |B_1| \int_0^r s^{n-1} \phi_{\alpha}(s)  ds. 
\]
Since $\phi_{\alpha}$ is decreasing, also $\Phi_{\alpha}$ is decreasing. In particular, since $\lim_{|x| \to \infty}  (I_{\alpha}f_E)(x) = 0$, $\phi_{\alpha}$ and $\Phi_{\alpha}$ are non-negative. Therefore, since $-\Delta (I_{\alpha +2}f_E)= (I_{\alpha}f_E)  $ by \eqref{Rieszlaplace}, we have that 
\begin{equation}
\label{standard}
\phi_{\alpha+2} '(s) = \frac{s}{n}\fint_{B_s} (I_{\alpha +2}f_E) (x) \, dx= -  \frac{s}{n}\fint_{B_s} (I_{\alpha }f_E)(x) \, dx = - \frac{s}{n} \Phi_{\alpha}(s) 
\end{equation}
and we conclude that $\phi_{\alpha+2}$ is decreasing and $(i)$ follows. Since  $\lim_{|x| \to \infty} (I_{\alpha +2}f_E)(x) = 0$, $\phi_{\alpha+2}$ is also non-negative.

Since $\Phi_{\alpha}$ is decreasing and non-negative we obtain from \eqref{standard} that  for every $s \in (0,1)$ it holds
\[
\phi_{\alpha+2}'(s) \leq - \frac{s}{n} \Phi_{\alpha}(1) . 
\]
Let $r \in (0,1)$. Integrate the above inequality over $[r,1]$ to obtain
\[
\phi_{\alpha+2}(1) -\phi_{\alpha+2}(r)= \int_r^1\phi_{\alpha+2}'(s)\, ds \leq - \frac{\Phi_{\alpha}(1)}{n}  \int_r^1 s\, ds =  - \frac{\Phi_{\alpha}(1)}{2n} (1- r^2).
\]
This implies
\begin{equation}
\label{standard2}
\phi_{\alpha+2}(r)   \geq \phi_{\alpha+2}(1)  +  \frac{\Phi_{\alpha}(1)}{2n} (1- r^2) \geq \frac{\Phi_{\alpha}(1)}{2n} (1- r^2)  .
\end{equation}
Notice that $\lim_{r \to 0}\phi_{\alpha+2}(r) = (I_{\alpha +2}f_E)(0)$ and therefore \eqref{standard2} implies 
\[
(I_{\alpha +2}f_E)(0) \geq  \frac{1}{2n}\Phi_{\alpha}(1)= \frac{1}{2n}  \fint_{B_1} (I_{\alpha }f_E)w_{\alpha}(x) \, dx 
\]
and $(iii)$ follows.

Claim $(ii)$ follows also from \eqref{standard2}, since
\[
\begin{split}
\int_{B_1} (I_{\alpha +2}f_E)(x) \, dx &=  \int_0^1 \int_{\partial B_r}  (I_{\alpha +2}f_E)(x) \, d \Ha^{n-1}(x) \,dr \\
&= n|B_1| \int_0^1 r^{n-1} \phi_{\alpha+2}(r)  dr \\
&\geq  \frac{|B_1|}{2} \, \Phi_{\alpha}(1) \int_0^1  (1- r^2) r^{n-1} \, dr \\
&=  \frac{|B_1|}{2n(n+2)} \Phi_{\alpha}(1) .
\end{split}
\]

\end{proof}

For the even dimensional cases we also need the fractional Laplacian with a fractional power $ \frac{1}{2}$. Suppose that $f:\R^n \to \R$ is an integrable function such that 
\begin{equation}
\label{int-condition}
\int_{\R^n} \frac{|f(x)|}{(1+ |x|)^{n+1}} \, dx < \infty.
\end{equation}
The $\frac{1}{2}$-Laplacian of $f$ is defined by
\[
(-\Delta)^{\frac{1}{2}}f(x) := \bar{\sigma}_{n} \int_{\R^n} \frac{f(x) - f(\xi)}{|x- \xi|^{n+1}} \, d \xi,
\]
where $\bar{\sigma}_{n} $ is some  normalization constant. The $\frac{1}{2}$-Laplacian can be considered as the inverse of the Riesz potential $I_1$, i.e.,  if $f:\R^n \to \R$ satisfies \eqref{int-condition} we have that 
\begin{equation}
\label{fraclaplacian}
(-\Delta)^{\frac{1}{2}}(I_1f) = f.
\end{equation}

The  $\frac{1}{2}$-Laplacian can also be obtained by extension problem. Denote $\R_+^{n+1} = \{ (x,z) \in \R^{n+1} \mid z>0 \}$.  Suppose that  $f:\R^n \to \R$ is a bounded function. Solve the problem
\[
\left\{
\begin{aligned}
-\Delta u(x,z) &= 0 &&\text{for }\,  (x, z) \in \R_+^{n+1}  \\
u(x, 0) &= f(x)  &&\text{for  }\, x\in \R^n. \\
\end{aligned}
\right.
\]
Then it holds 
\begin{equation}
\label{extension}
- u_z(x,0)= (-\Delta)^{\frac{1}{2}}f(x).
\end{equation}
For further details and results about the fractional Laplacian with a general fractional power we refer to \cite{CS} and \cite{Sil}.

\begin{proof}[\textbf{Proof of the Theorem \ref{mainthm} in dimensions $n \geq 4$.}]

The proof begins exactly as in the case $n=3$. We consider the problem in the form  \eqref{scale}. We  may assume that $E \subset \R^n$, $|E|= |B_1|$,  is such that 
\[
\gamma(E) = \int_{B_1} \frac{n-1}{|x|} dx - \int_{E} \frac{n-1}{|x|} dx.
\] 
We need to show that
\[
 \text{NL}(B_1) -\text{NL}(E) \lesssim P(B_1) -P(E)
\]  
where  $\text{NL}(E) := \int_{\R^n}\int_{\R^n} \frac{\chi_E(x)\chi_E(y)}{|x-y|^{n-2}} \, dy\,  dx$. Recall that ''$\lesssim$'' means that the inequality is true up to a positive constant. In fact, in the beginning of the proof in $n=3$ we have  already proven that 
\begin{equation}
\label{alrknown}
 \text{NL}(B_1) -\text{NL}(E) \leq  2 \int_{B_1} v(y)\, dy,
\end{equation}
where  
\begin{equation}\label{def_v}
v(y) = \int_{\R^n} \frac{ (\chi_{B_1} - \chi_{E})(x)}{|x-y|^{n-2}}  \, dx,
\end{equation}
which solves 
\[
- \Delta v =   \tilde{c}_n \, (\chi_{B_1} - \chi_{E})
\]
for some $ \tilde{c}_n >0$.  Define 
 $\phi : (0, \infty) \to \R$
\begin{equation}
\label{def_phi}
\phi(r) :=  \fint_{\partial B_r} v(y) \, d \Ha^{n-1}(y).
\end{equation}
For every $r > 0$ it holds 
\[
\phi'(r) = \frac{r}{n} \fint_{B_r} \, \Delta v (y) dy = -\frac{\tilde{c}_n  r}{n}\fint_{B_r} (\chi_{B_1} - \chi_{E})(y) \, dy \leq 0
\]
and therefore $\phi$ is decreasing.

We divide the proof in two cases whether the dimension $n$ is odd or even.

\smallskip

\textbf{ The case when $n \geq 5$ is odd :}

Denote 
\[
f_E = \chi_{B_1}- \chi_E. 
\]
Consider the functions $I_2 f_E,  I_4 f_E, I_6  f_E,  \dots, I_{n-1} f_E$, where  $I_{\alpha}$ is the Riesz potential defined in \eqref{Riesz}.  In other words
\[
(I_{2k}f_E)(y) =\sigma_{n , 2k}  \int_{\R^n} \frac{ (\chi_{B_1} - \chi_{E})(x)}{|x-y|^{n-2k}}  \, dx \qquad \text{for }\, k = 1,2, \dots \frac{n-1}{2}.
\]
By \eqref{Rieszlaplace} we have that $-\Delta (I_{2k +2}f_E) = I_{2k}f_E $ for every $k = 1, 2, \dots, \frac{n-3}{2}$ and $I_2 f_E =  \sigma_{n , 2} \, v$ , where $v$ is defined in \eqref{def_v}. Define also $\phi_{2}, \phi_{4}, \dots, \phi_{n-1}: (0, \infty) \to \R$ such that
\[
\phi_{2k}(r) :=  \fint_{\partial B_r} (I_{2k}f_E)(y) \, d \Ha^{n-1}(y)\qquad \text{for }\, k = 1,2, \dots \frac{n-1}{2}.
\]
Notice that $\phi_{2}=  \sigma_{n , 2} \, \phi$, where  $\phi$ is defined in \eqref{def_phi}. Since $\phi$ is decreasing also $\phi_{2}$ is decreasing. 

We may use Lemma \ref{induction} $(i)$ and $(ii)$  iteratively for  $I_2 f_E, I_4 f_,  \dots, I_{n-3} f_E$ to obtain 
\[
 \fint_{B_1}(I_2 f_E)(y)\, dy \lesssim  \fint_{B_1}  (I_{n-3} f_E)(y)\, dy
\]
and that $\phi_{n-3}$ is decreasing. In other words 
\begin{equation}
\label{after_induc}
 \int_{B_1} v(y)\, dy=  \int_{B_1} \int_{\R^n} \frac{(\chi_{B_1} - \chi_{E})(x)}{|x-y|^{n-2}} \, dx dy \lesssim  \int_{B_1} \int_{\R^n}  \frac{(\chi_{B_1} - \chi_{E})(x)}{|x-y|^{3}} \, dx dy.
\end{equation}
Finally we use Lemma \ref{induction} $(iii)$ with $\alpha = n -3$, i.e., to functions $I_{n-3} f_E$ and $I_{n-1} f_E$, to obtain 
$ \fint_{B_1}  (I_{n-3} f_E) (y) \, dy \leq 2n\,  (I_{n-1} f_E)(0)$. This, in turn, implies 
\[
 \int_{B_1}  \int_{\R^n}   \frac{(\chi_{B_1} - \chi_{E})(x)}{|x-y|^{3}} \, dxdy \lesssim  \int_{\R^n} \frac{(\chi_{B_1} - \chi_{E})(x)}{|x|} \, dx .
\]
The previous inequality and \eqref{after_induc} imply
\[
 \fint_{B_1} v(y)\, dy \lesssim \int_{\R^n} \frac{(\chi_{B_1} - \chi_{E})(x)}{|x|} \, dx  = \frac{\pot(E)}{n-1} .
\]
Combining the above estimate with \eqref{alrknown} and \eqref{pot_form}, we get
\[
 \text{NL}(B_1) -\text{NL}(E) \leq  2 \int_{B_1} v(y)\, dy \lesssim \pot(E) \lesssim  P(E) - P(B_1)
\]
and the claim follows.

\smallskip

\textbf{ The case when $n \geq 4$ is even:}

Denote 
\[
f_E = \chi_{B_1}- \chi_E
\]
and consider functions $I_2 f_E,  I_4 f_E, I_6  f_E,  \dots, I_{n-2} f_E$, i.e., 
\[
( I_{2k} f_E)(y) =\sigma_{n , 2k}  \int_{\R^n} \frac{ (\chi_{B_1} - \chi_{E})(x)}{|x-y|^{n-2k}}  \, dx \qquad \text{for }\, k = 1,2, \dots \frac{n-2}{2}
\]
and functions $\phi_{2}, \phi_{4}, \dots, \phi_{n-2}: (0, \infty) \to \R$ such that
\[
\phi_{2k}(r) :=  \fint_{\partial B_r} ( I_{2k} f_E)(y) \, d \Ha^{n-1}(y) \qquad \text{for }\, k = 1,2, \dots \frac{n-2}{2}.
\]
Again $\phi_{2}$ is decreasing and we may use Lemma \ref{induction} $(i)$ and $(ii)$  iteratively for   $I_2 f_E,  I_4 f_E,   \dots, I_{n-2} f_E$ to obtain 
\[
 \fint_{B_1} (I_2 f_E)(y)\, dy \lesssim  \fint_{B_1} (I_{n-2} f_E)(y)\, dy,
\]
and that $\phi_{n-2}$ is decreasing (and non-negative).  In other words 
\begin{equation}
\label{even-ind}
 \int_{B_1} v(y)\, dy=  \int_{B_1}  \int_{\R^n} \frac{(\chi_{B_1} - \chi_{E})(x)}{|x-y|^{n-2}} \, dxdy \lesssim    \int_{B_1} (I_{n-2} f_E)(y)\, dy.
\end{equation}
To conclude the proof we need to show that 
\begin{equation}
\label{to-show}
 \int_{B_1} (I_{n- 2}f_E)(y)\, dy \lesssim  (I_{n- 1}f_E)(0) .
\end{equation}
The inequalities   \eqref{even-ind}  and \eqref{to-show} would imply 
\[
\int_{B_1} v(y)\, dy  \lesssim  (I_{n- 1}f_E)(0) = \sigma_{n, n-1} \int_{\R^n} \frac{(\chi_{B_1} - \chi_{E})(x)}{|x|} \, dx  \lesssim \pot(E)
\]
and the claim would follow from  \eqref{alrknown} and \eqref{pot_form}.

To prove \eqref{to-show} we use the $\frac{1}{2}$-Laplacian. By the semigroup property \eqref{semigroup} it holds $ I_{n-1}(f_{E}) = I_1 (I_{n-2}f_{E})$ and therefore by \eqref{fraclaplacian} we have
\[
 (- \Delta)^{\frac{1}{2}}(I_{n- 1} f_E) = I_{n-2}f_{E}.
\]
Suppose that  $u: \R_+^{n+1} \to \R$ is the harmonic extension of $I_{n- 1}f_E$, i.e.,
\[
\left\{
\begin{aligned}
-\Delta u(y, z) &= 0 &&\text{for }\, (y, z) \in \R_+^{n+1}  \\
u(y, 0) &= I_{n- 1}f_E(y)  &&\text{for  }\, y\in \R^n. \\
\end{aligned}
\right.
\]
Then $-  u_z(y,0) =  (- \Delta)^{\frac{1}{2}}(I_{n- 1}f_E)(y) = (I_{n- 2}f_E)(y)$ for $y \in \R^n$.

Extend $u$ to $\R^{n+1}$ by reflection, i.e., define $\tilde{u} : \R^{n+1} \to \R$ such that
\[
\left\{
\begin{aligned}
\tilde{u}(y, z) &=  u(y, z)  &&\text{for }\, z\geq 0  \\
\tilde{u}(y, z) &=  u(y, -z)   &&\text{for  }\, z<0. \\
\end{aligned}
\right.
\]
Denote $Y= (y, z) \in \R^{n+1}$ and denote by $\hat{B}_r$ the $(n+1)$-dimensional ball with radius $r$ and $|\hat{B}_1|= \omega_{n+1}$. Define further $\phi_{\tilde{u}}: (0, \infty) \to \R$ such that 
\[
\phi_{\tilde{u}}(r) := \fint_{\partial \hat{B}_r} \tilde{u}(Y) \, d \Ha^{n}(Y).
\] 
By the divergence theorem we get
\begin{equation} \label{divergence}
\begin{split}
\phi_{\tilde{u}}'(r) &=\fint_{\partial \hat{B}_r} \langle \nabla \tilde{u}(Y), \frac{Y}{r} \rangle  \, d \Ha^{n}(Y)\\
&=  \frac{2}{(n+1)\, \omega_{n+1}r^n} \left( \int_{\partial ( \hat{B}_r \cap  \R_+^{n+1})} \langle \nabla u(Y), \nu \rangle  \, d \Ha^{n}(Y) -   \int_{   \hat{B}_r  \cap \partial \R_+^{n+1} } -  u_z(Y) \, d \Ha^{n}(Y) \right) \\
&= \frac{2 }{(n+1)\, \omega_{n+1}r^n} \left( \int_{\hat{B}_r \cap  \R_+^{n+1}} \Delta u (Y)  \, d Y -  \int_{ B_r}  (- \Delta)^{\frac{1}{2}}(I_{n- 1}f_E)(y)  \, dy \right) \\
&= - 2 \, \frac{\omega_{n}}{(n+1)\, \omega_{n+1}}\fint_{ B_r} (I_{n- 2}f_E)(y)   \, dy,
\end{split}
\end{equation}
where $\omega_n$ is the measure of the $n$-dimensional unit ball. Since $\phi_{n-2}(r)= \fint_{\partial B_r} ( I_{n-2} f_E)(y) \, d \Ha^{n-1}(y)$ is decreasing and non-negative, also $r \mapsto \fint_{ B_r} ( I_{n-2} f_E)(y) \, dy$ is decreasing and non-negative. Therefore \eqref{divergence} yields 
\[
\phi_{\tilde{u}}'(r) \leq -  \tilde{C}_n   \fint_{ B_1} ( I_{n-2} f_E)(y)   \, dy,
\]
for $\tilde{C}_n  = \frac{2\, \omega_{n}}{(n+1)\, \omega_{n+1}}$. Integrate the above inequality over $[R,1]$ to obtain
\[
\phi_{\tilde{u}}(1) - \phi_{\tilde{u}}(R) \leq  - \tilde{C}_n  \left(  \fint_{ B_1} ( I_{n-2} f_E)(y)   \, dy \right)  (1 -R) .
\]
Since $\lim_{R \to 0} \phi_{\tilde{u}}(R) = \tilde{u}(0) = (I_{n- 1}f_E)(0) $ we obtain from the above estimate that
\[
(I_{n- 1}f_E)(0) \geq \phi_{\tilde{u}}(1) + \tilde{C}_n \fint_{ B_1} (I_{n- 2}f_E)(y)   \, dy  \geq \tilde{C}_n  \fint_{ B_1} (I_{n- 2}f_E)(y)   \, dy ,
\] 
which proves \eqref{to-show}. 
\end{proof}

\textbf{Acknowledgements.} The research of  was partially supported by the 2008 ERC Grant no.226234 ''Analytic Techniques for Geometric and Functional Inequalities''

\end{document}